\documentclass[preprint,12pt]{elsarticle}

\usepackage{amssymb}
\usepackage{hyperref}
\usepackage{amsmath}
\usepackage{fixltx2e}
\usepackage{footnote}
\usepackage[usenames, dvipsnames]{color}
\usepackage{multicol}
\usepackage{booktabs}
\usepackage[flushleft]{threeparttable}
\usepackage[font=small,labelfont=bf]{caption}
\newtheorem{thm}{Theorem}
\newtheorem{lemma}[thm]{Lemma}
\newdefinition{rmk}{Remark}
\newproof{proof}{Proof}

\journal{---}

\begin{document}

\begin{frontmatter}


 \title{Zeroth-order general Randi\' c index of $k$-generalized quasi trees}
%%\tnotetext[label1]{}
\author{Muhammad Kamran Jamil\fnref{label2}}
\ead{m.kamran.sms@gmail.com}
\fntext[label2]{Corresponding author}
\address{Department of Mathematics, Riphah Institute of Computing and Applied Sciences (RICAS),\\
Riphah International University, 14 Ali Road, Lahore, Pakistan.}
%%\fntext[label3]{xxx}

\author{Ioan Tomescu}
\address{Faculty of Mathematics and Computer Science, University of Bucharest, Romania.}
\ead{ioan@fmi.unibuc.ro}

%\author{\corref Muhammad Kamran Jamil\footnote{ Abdus Salam School of Mathematical Sciences, Government College University, Lahore.  \ Email: {\tt m.kamran.sms@gmail.com }}\footnote{Corresponding Author.}}

\address{}

\begin{abstract}
For a simple graph $G(V,E)$, the zeroth-order general Randi\' c index is defined as $^0R_{\alpha}(G)=\sum_{v\in V(G)}d(v)^{\alpha}$, where $d(v)$ is the degree of the vertex $v$ and $\alpha\ne0$ is a real number. The $k$-generalized quasi-tree is a connected graph $G$ with a subset $V_k\subset V(G)$, where $|V_k|=k$ such that $G-V_k$ is a tree,  but for any subset $V_{k-1}\subset V(G)$ with cardinality $k-1$, $G-V_{k-1}$ is not a tree. In this paper, we characterize the extremal $k$-generalized quasi trees with the minimum and maximum values of the zeroth-order general Randi\' c index for $\alpha\neq 0$.
\end{abstract}


\end{frontmatter}


\section{Introduction}

Let $G=(V(G),E(G))$ be a simple connected graph, where $V(G)$ and $E(G)$ represent the sets of vertices and edges, respectively. The Randi\' c index \cite{r} introduced in 1975, is defined as follows:
\[R(G)=\sum_{uv\in E(G)}(d(u)d(v))^{-1/2},\]
where $d(v)$ is the degree of the vertex $v$ in $G$.\\

Li et al. proposed the general Randi\' c index by replacing the  exponent $-1/2$ by an arbitrary real number $\alpha$. This index is defined as
\[R_{\alpha}(G)=\sum_{uv\in E(G)}(d(u)d(v))^{\alpha}.\]

The zeroth-order Randi\' c index, defined by Kier et al. \cite{kh}, is
\[^0R_{-1/2}(G)=\sum_{v\in V(G)}{d(v)}^{-1/2}.\]

The first Zagreb index was introduced by Gutman et al. in 1972 \cite{grt}  and it is defined as
\[^0R_{2}(G)=\sum_{v\in V(G)}{d(v)}^{2}.\]

The common generalization of the first Zagreb index and the zeroth-order Randi\' c index was made by Li et al. \cite{lz1}. He proposed the zeroth-order general Randi\' c index $^0R_{\alpha}$ by
\[^0R_{\alpha}(G)=\sum_{v\in V(G)}{d(v)}^{\alpha}.\]

The above mentioned topological indices have been closely correlated with many physical and chemical properties of the molecules such as boiling point, calculated surface, molecular complexity, heterosystems, chirality, e.g. More information on these indices can be obtained from \cite{gja,hly,hly1,kh1,kh2,pl}.\\

A graph $G$ is called a quasi-tree, if there exists a vertex $z\in V(G)$ such that $G-z$ is a tree and such a vertex is called a quasi vertex. As deletion of any vertex with degree one will deduce another tree it follows that any tree is a quasi tree. A graph $G$ is called $k$-generalized  quasi tree if there exists a subset $V_k\subset V(G)$ with cardinality $k$ such that $G-V_k$ is a tree but for any subset $V_{k-1}\subset V(G)$ with cardinality $k-1$, $G-V_{k-1}$ is not a tree. The vertices of $V_k$ are also called quasi vertices (or $k$-quasi vertices). To draw a $k$-generalized quasi tree we need at least $k+2$ vertices. We call any tree a trivial quasi tree and other quasi trees are called non-trivial quasi trees. We denote the class of $k$-generalized quasi trees of order $n$ by $T_k(n)$.\\

All graphs considered in this paper are undirected, finite, simple and connected. For terminology and notation not defined here we refer \cite{bm}. Let $G$ and $H$ be two vertex disjoint graphs. $G+H$ denotes the join graph of $G$ and $H$ with vertex set $V(G+H)=V(G)\cup V(H)$ and the edge set $E(G+H)=E(G)\cup E(H)\cup \{uv|v\in V(G),v\in V(H)\}$.   $ S_n$ and $P_n$ represent the star and the path of order $ n $, respectively. $S_{p,q}(u,v)$ denotes the bistar of order $p+q$, which is a tree consisting of  two adjacent vertices $u$ and $v$, such that $u$ is adjacent to $p-1$ pendant vertices and $q$ is adjacent to $q-1$ pendant vertices. If $G$ and $H$ are vertex disjoint graphs and $u,v\in V(H)$, $G\bullet _{u,v}H$ represents the graph having vertex set $V(G)\cup V(H)$ obtained by joining every vertex of $G$  to vertices $u$ and $v$ of $H$.\\

Akhter et al. \cite{ajt} found the extremal first and second Zagreb indices of $k$-generalized quasi trees. Qiao \cite{q} determined the extremal $k$-generalized quasi trees, for $k=1$, with the minimum and maximum values of the zeroth-order general Randi\' c index.  In this paper, we characterize the extremal $k$-generalized quasi trees of order $n$ with the maximum and minimum values of the zeroth-order general Randi\' c index for $\alpha\neq 0$. Our results extend the results of Akhter and Qiao.

\section{Results and Discussion}
In this section, first we will discuss some auxiliary lemmas which will be helpful to prove main results. 

\begin{lemma}{\rm \cite{lz}}\label{z}
	Among all trees with $n$ vertices, the trees with extremal zeroth-order general Randi\' c index are listed in the following table:
	\begin{table}[h!]
		\begin{tabular}{c|c|c}
			& $\alpha <0$ or $\alpha >1$ &$0< \alpha <1$\\
			\hline
			minimum & the path $P_n$ & the star $S_n$\\
			\hline
			second minimum&trees with $[3,2^{n-4},1^3]$&the double star $S_{n-2,2}$\\
			\hline
			third minimum & trees with $[3^2,2^{n-6},1^4]$& the double star $S_{n-3,3}$\\
			\hline
			maximum& the star $S_n$& the path $P_n$\\
			\hline
			second maximum&the double star $S_{n-2,2}$& trees with $[3,2^{n-4},1^3]$ \\
			\hline
			third maximum& the double star $S_{n-3,3}$& trees with $[3^2,2^{n-6},1^4]$\\
			\hline
		\end{tabular}
	\end{table}

\end{lemma}

\begin{lemma}\label{max}
	If $u,v\in V(G)$ such that $uv\notin E(G)$, then for\\
	$\alpha<0$
		\[^0R_{\alpha}(G+uv)<{^0R}_{\alpha}(G), \]	
	and for $ \alpha>0$	
	\[^0R_{\alpha}(G+uv)>{^0R}_{\alpha}(G).\]
\end{lemma}

\begin{lemma}\label{deg}
	Let $G\in T_k(n)$. If $^0R_{\alpha}(G)$ is minimum $($maximum$)$ and $z$ is a quasi vertex of G, then $d(z)=n-1$ for  $\alpha<0$ $(\alpha>0, respectively)$.
\end{lemma}

\begin{proof}
	Let $G\in T_k(n)$, $^0R_{\alpha}(G)$ be minimum (maximum) and $z$ be a quasi vertex of $G$. Suppose on contrary $d(z)<n-1$, then there is a vertex $x\in V(G)$ such that $xz\notin E(G)$. Now $G+xz$ is also in $T_k(n)$ and $^0R_{\alpha}(G+xz)<^0R_{\alpha}(G)$ for $\alpha <0$ ($^0R_{\alpha}(G+xz)>^0R_{\alpha}(G)$ for $\alpha >0$), a contradiction, hence $d(z)=n-1$.
\end{proof}

%\begin{lemma}\label{trans}
%Let $G$ be a graph and $u,w,x\in V(G)$ such that $d(u)< d(w)$ and $xu\in E(G)$. If $\alpha<0$ we obtained a graph $G'=G-xu+xw$ then\\
%\[^0R_{\alpha}(G')>^0R_{\alpha}(G)\]
%\end{lemma}
%\begin{proof}
%	By definition
%	\begin{align*}
%		^0R_{\alpha}(G')-^0R_{\alpha}(G)&=(d(u)-1)^{\alpha}+(d(w)+1)^{\alpha}-d(u)^{\alpha}-d(w)^{\alpha}\\
%		&>0
%	\end{align*}
%	Since $f(x)=x^{\alpha}-(1+x)^{\alpha}$ is a decreasing function for $\alpha<0$.
%\end{proof}
\begin{lemma}\label{f}
Let $f(x)=x^{\alpha}-(x+1)^{\alpha}$, where $x>0$. $f(x)$ is strictly increasing for $0<\alpha <1$ and strictly decreasing for $\alpha<0$ or $\alpha >1$.
\end{lemma}
\begin{lemma}\label{trans}
	Let $G$ be a graph, and $u,v$ and $w$ be three vertices of $G$ such that $uw\notin E(G)$, $vw\in E(G)$ and $d(u)\geq d(v)$.
Let $G'=G+uw-vw$. If $\alpha<0$ or $\alpha >1$ then $^0R_{\alpha}(G')>^0R_{\alpha}(G)$ and if $0<\alpha <1$ then  $^0R_{\alpha}(G')<^0R_{\alpha}(G).$ 
\end{lemma}
\begin{proof}
	Let $d(u)=x$ and $d(v)=y$. We obtain 
	$^0R_{\alpha}(G')-^0R_{\alpha}(G)=(x+1)^{\alpha}+(y-1)^{\alpha}-x^{\alpha}-y^{\alpha}=f(y-1)-f(x)$, where $f(x)=x^{\alpha}-(x+1)^{\alpha}$.
	$f(x)$ is a strictly decreasing function for $x>0$ and $\alpha<0$ or $\alpha >1$.
Since $y-1<x$ it follows that $^0R_{\alpha}(G') > ^0R_{\alpha}(G)$.  If $0<\alpha <1$ the proof is similar.
\end{proof}

\begin{lemma}\label{d2}
	Let $G\in T_k(n)$. If $^0R_{\alpha}(G)$ is maximum $($minimum$)$ then there exists a spanning subgraph $H$ of $G$ such that  $^0R_{\alpha}(G)\le  ^0R_{\alpha}(H)$
$( ^0R_{\alpha}(G)\ge \\ ^0R_{\alpha}(H))$
and  for any quasi vertex $z$ of G we have $d_{G}(z)\ge d_{H}(z)=2$ and $z$ is adjacent in $H$ to other two vertices which are not quasi vertices for $\alpha<0$ $(\alpha>0, respectively)$.
\end{lemma}
\begin{proof}
	By definition of a $k$-generalized quasi tree, there exists a subset $X\subset V(G)$ of cardinality $k$ such that $G-X$ is a tree and for any $Y\subset V(G)$ and $|Y|<k$, $G-Y$ is not a tree. It follows that $d(z)\ge 2$ for any vertex $z\in X$.  If $m$ denotes the number of edges of $G$, then $m\ge 2k+n-k-1=n+k-1$ and equality holds if and only if $d(z)=2$ for any vertex $z\in X$ and no two vertices in $X$ are adjacent.
By Lemma \ref{max}, by deleting some edges it follows the existence of the graph $H$, which is not necessarily in $T_k(n)$.  
\end{proof}

\begin{lemma}\label{aq}
	Let $n,x_i (1\leq i\leq n),p,m\ge 1$ be integers, $\alpha$ be any real number such that $\alpha\notin \{0,1\}$ and $x_1+x_2+\cdots+x_n=p$. \\
{\rm a)}The function $f(x_1,x_2,\ldots,x_n;p)=\sum_{i=1}^{n}x_i^{\alpha}$ is minimum for $\alpha<0$ or $\alpha >1$ $($maximum for $0<\alpha<1$, respectively$)$ if and only if $x_1,x_2,\ldots,x_n$ are almost equal, or $|x_i-x_j|\le1$ for every $i,j=1,2,\ldots,n$. \\
{\rm b)}If $x_1\geq x_2 \geq m$, the maximum of the function $f(x_1,\ldots ,x_n)$ is reached for $\alpha<0$ or $\alpha >1$ $($minimum for $0<\alpha<1$, respectively$)$ only for $x_1=p-m-n+2, x_2=m,x_3=x_4=\ldots =x_n=1$. The second maximum $($the second minimum, respectively$)$ is attained only for $x_1=p-m-n+1, x_2=m+1,x_3=x_4=\ldots =x_n=1$.
\end{lemma}
\begin{proof}
We shall consider only the case $\alpha <0$ or $\alpha >1$, the proof in the other case being similar.\\
	a) The function  $f(x)=x^{\alpha}-(1+x)^{\alpha}$ is a strictly decreasing function for $x>0$ and $\alpha>1$ or $\alpha <0$. If $x\geq y+2>0$ we deduce $x-1>y$, which implies $f(x-1)<f(y)$, or $x^{\alpha}+y^{\alpha}>(x-1)^{\alpha}+(y+1)^{\alpha}$.
	It follows that $f(x_1,x_2,\cdots,x_n;p)=\sum_{i=1}^{n}x_i^{\alpha}$ is minimum if and only if  $x_1,x_2,\ldots,x_n$ are almost equal.\\
b) If $x\geq y\geq 2$ then $x>y-1$, which implies $f(y-1)>f(x)$, or $(x+1)^{\alpha}+(y-1)^{\alpha}>x^{\alpha}+y^{\alpha}$.
\end{proof}

\section{Case $\alpha<0$}
\begin{thm}
	Let $G\in T_k(n)$, where $k\geq 1$ and $n\ge 3$. For $\alpha<0$ we have
\[^0R_{\alpha}(G)\ge k(n-1)^{\alpha}+2(k+1)^{\alpha}+(n-k-2)(k+2)^{\alpha}\]
and equality holds if and only if $G=K_k+P_{n-k}$.
\end{thm}

\begin{proof}
	Suppose that $G\in T_k(n)$ has minimum $^0R_{\alpha}(G)$. Let $V_k\subset V(G)$ be the set of $k$-quasi vertices. As $^0R_{\alpha}(G+uv)<^0R_{\alpha}(G)$ for any $uv\notin E(G)$, this implies that $V_k$ forms a complete graph in $G$. Then by Lemma \ref{deg} we have $G=K_k+T_{n-k}$, where $T_{n-k}$ is a tree of order $n-k$.
We can write:
	\begin{align*}
	^0R_{\alpha}(G)&=^0R_{\alpha}(K_k+T_{n-k})\\
	&=\sum_{v\in V(K_k)}(d(v)+n-k)^{\alpha}+ \sum_{v\in V(T{n-k})}(d(v)+k)^{\alpha}\\
	&=k(n-1)^{\alpha}+ \sum_{v\in V(T_{n-k})}(d(v)+k)^{\alpha}.
	\end{align*}
We get $$\sum_{v\in V(T_{n-k})}(d(v)+k)=2(n-k-1)+k(n-k).$$ By Lemma \ref{aq},
 $\sum_{v\in V(T{n-k})}(d(v)+k)^{\alpha}$ is minimum if and only if the degrees of $T_{n-k}$ are almost equal. Since every tree has at least two vertices of degree one, it follows that the minimum of this sum is reached if and only if $T_{n-k}$ has two vertices of degree one and $n-k-2$ vertices of degree 2, or $T_{n-k}=P_{n-k}$. Finally,
	\[ ^0R_{\alpha}(G)  \ge k(n-1)^{\alpha}+2(k+1)^{\alpha}+(n-k-2)(k+2)^{\alpha}.\]
	Equality holds if and only if $G=K_k+P_{n-k}$.
\end{proof}

\begin{thm}
	Let $G\in T_k(n)$, where $n\ge 3$ and $k\geq 1$. If $\alpha<0$ we have:\\
{\rm a)} If $k=1$ then
	\[^0R_{\alpha}(G)\le (n-1)^{\alpha}+2^{\alpha+1}+n-3\]
	and equality holds if and only if $G={K_1}\bullet _{u,v}S_{n-1}$, where $u$ is the center of $S_{n-1}$ and $v$ is a pendant vertex of $S_{n-1}$.\\
{\rm b)} If $n\ge 4$ and $k\geq 2$ then \[^0R_{\alpha}(G)\le (n-2)^{\alpha}+k2^{\alpha}+(k+2)^{\alpha}+n-k-2\]
and equality holds if and only if $G=\overline{K_k}\bullet _{u,v}S_{n-k-2,2}(u,v)$, where $u$ and $v$ are vertices of degree $n-k-2$ and $2$ of $S_{n-k-2,2}(u,v)$, respectively.

\end{thm}
\begin{proof}
	Suppose that $G\in T_k(n)$ has maximum $^0R_{\alpha}(G)$. Let $V_k\subset V(G)$ be the set of $k$-quasi vertices. The graph $G-V_k$ is a tree of order $n-k$, denoted by $T_{n-k}$. As $^0R_{\alpha}(G-uv)>^0R_{\alpha}(G)$ for any $uv\in E(G)$, and
by Lemmas \ref{trans} and  \ref{d2} we deduce the existence of a graph $F$ with $V(F)=V(G)$, $ ^0R_{\alpha}(G)\le  ^0R_{\alpha}(F)$ and
 such that in $F$ we have: $V_k$ forms an empty graph, i.e., it induces $\overline{K_k}$, every quasi vertex of $G$ has degree 2 and quasi vertices have common neighbors $y_1,y_2\in V(G)$,
where $y_1$ is a vertex of maximum degree in $T_{n-k}$ and $y_2$ is a vertex of maximum degree in $T_{n-k}-y_1$.
We can represent the graph $F$ as $F=\overline{K_k}\bullet _{y_1,y_2} T_{n-k}$. We deduce:
	
	$$^0R_{\alpha}(F)=	^0R_{\alpha}(\overline{K_k}\bullet _{y_1,y_2}T_{n-k})=\sum_{v\in V(\overline{K_k}\bullet _{y_1,y_2}T_{n-k})}d(v)^{\alpha}$$
		$$=\sum_{v\in V(\overline{K_k})}d(v)^{\alpha}+ \sum_{v\in V(T{n-k})\hfill \atop v\ne y_1, v\ne y_2 } d(v)^{\alpha}+ (d(y_1)+k)^{\alpha} + (d(y_2)+k)^{\alpha}.$$
%	equality holds if and only if all the $k$-quasi vertices have common neighbors $y_1$ and $y_2$ such that $d(v)\le d(y_1),d(y_2)$ for all $v\in G-V_k$.
%	\begin{align*}
%		&\le k2^{\alpha}+   \sum_{v\in V(T{n-k})\hfill \atop v\ne y_1, v\ne y_2 } d(v)^{\alpha}+(k+2)^{\alpha}+(n-2)^{\alpha}
%	\end{align*}
%	equality holds if and only if $d(y_1)=2$ and $d(y_2)=n-k-2$ and from 
	%&\le (n-2)^{\alpha}+(k+2)^{\alpha}+k2^{\alpha}+(n-k-2).\\
	We have $$ \sum_{v\in V(T{n-k})\hfill \atop v\ne y_1, v\ne y_2 } d(v)+ d(y_1)+k + d(y_2)+k=2n-2.$$
	By Lemma \ref{aq}, the sum
\begin{equation}
 \sum_{v\in V(T{n-k})\hfill \atop v\ne y_1, v\ne y_2 } d(v)^{\alpha}+ (d(y_1)+k)^{\alpha} + (d(y_2)+k)^{\alpha}
\end{equation}
is maximum only if $T_{n-k}=S_{n-k}$ and $y_1$ and $y_2$ are the center and a pendant vertex of $S_{n-k}$, respectively. For $k=1$ this graph is a $k$-generalized quasi-tree, but for $k\ge 2$ this property is no longer valid. We must consider the second maximum of (1). This time $F\in T_k(n), G=F$ and $T_{n-k}=S_{n-k-2,2}(u,v)$, $y_1=u$ and $y_2=v$.
The conclusion follows.
\end{proof}

\section{Case $\alpha \ge1$}
%	\begin{lemma}\label{d2min}
%		Let $G\in T_k(n)$, $^0R_{\alpha}(G)$ is minimum and $z$ is an quasi vertex of G, then d(z)=2 for $\alpha \ge 1$.
%	\end{lemma}
%	\begin{proof}
%		By definition of a $k-$quasi tree, there exists a subset $X\subset V(G)$ of cardinality $k$ such that $G-X$ is a tree and for any $Y\subset V(G)$ and $|Y|<k$, $G-Y$ is not a tree. It follows that $d(z)\ge 2$ for any vertex $z\in X$. Let $m$ denotes the number of edges of $G$, by Lemma \ref{max} $^0R_{\alpha}(G)$ is minimum when $m$ is minimum i.e. $m\ge 2k+n-k-1=n+k-1$ equality holds if and only if $d(z)=2$ for any vertex $z\in V(G)$.
%	\end{proof}

	\begin{thm}
		Let $G\in T_k(n)$, $k\geq 1$ and $n\ge 3$, then
		for $\alpha =1$
		\[2(n+k-1)\le ^0R_{\alpha}(G)\le 2n(k+1)-k(k+3)-2. \]
		Left equality holds if and only if $G$ consists of $\overline{K_k}$, a tree $T_{n-k}$ of order $n-k$, every vertex of $\overline{K_k}$ being adjacent to two arbitrary vertices of $T_{n-k}$ such that the resulting graph belongs to $T_k(n)$ and the right equality holds if and only if $G=K_k+T_{n-k}$.
	\end{thm}

\begin{proof}
	For $\alpha=1$ we have $^0R_{\alpha}(G)=\sum_{v\in V(G)}d(v)=2|E(G)|\ge 2(n+k-1)$ and equality holds if and only if the degree of every quasi vertex is two. Hence, the left hand inequality.\\
Similarly, $|E(G)|$ is maximum only if $G=K_k+T_{n-k}$ and the right hand inequality follows.

\end{proof}

%\begin{lemma}\label{degmax}
%	Let $G\in T_k(n)$,$n\ge 3$, $^0R_{\alpha}(G)$, where $\alpha >1$ is as maximum and $z$ is an quasi vertex of G then,
%	\begin{enumerate}
%		\item d(z)=n-1
%	\end{enumerate}
%\end{lemma}
%\begin{proof}
%	Let $G\in T_k(n)$, $^0R_{\alpha}(G)$ is maximum and $z$ be an quasi vertex of $G$. Suppose on contrary $d(z)<n-1$, then there is a vertex $x\in V(G)$ such that $xz\notin E(G)$. Now $G+xz$ is also in $T_k(n)$ and $^0R_{\alpha}(G+xz)>^0R_{\alpha}(G)$, a contradiction, hence $d(z)=n-1$.
%\end{proof}

\begin{thm}
	Let $G\in T_k(n)$ and $k\geq 1$, $n\ge 3$, $\alpha>1$ then
	\[(n-2k+2)2^{\alpha}+(2k-2)3^{\alpha}\le^0R_{\alpha}(G)\le (k+1)(n-1)^{\alpha}+(n-k-1)(k+1)^{\alpha}.\]
	The upper bound is an equality if and only if $G=K_k+S_{n-k}$.
\end{thm}

\begin{proof}
	Suppose that $G\in T_k(n)$ has maximum $^0R_{\alpha}(G)$. Let $V_k\subset V(G)$ be the set of $k$-quasi vertices. As $^0R_{\alpha}(G+uv)>^0R_{\alpha}(G)$ for any $uv\notin E(G)$, this implies that $V_k$ induces a complete subgraph in $G$. Then by Lemma \ref{deg} we have $G=K_k+T_{n-k}$, where $T_{n-k}$ is a tree of order $n-k$.
It follows that:
	\begin{align*}
	^0R_{\alpha}(G)&= {^0}R_{\alpha}(K_k+T_{n-k})\\
	&=\sum_{v\in V(K_k)}(d(v)+n-k)^{\alpha}+ \sum_{v\in V(T_{n-k})}(d(v)+k)^{\alpha}\\
	&=k(n-1)^{\alpha}+ \sum_{v\in V(T_{n-k})}(d(v)+k)^{\alpha}\\
	&\le  (k+1)(n-1)^{\alpha}+(n-k-1)(k+1)^{\alpha}.
	\end{align*}
	By Lemma \ref{aq} the upper bound is an equality if and only if $T_{n-k}=S_{n-k}$, i.e., $G=K_k+S_{n-k}$.
\par
Suppose now that $^0R_{\alpha}(G)$ is minimum. By Lemma \ref{d2} there exists a spanning subgraph $H$ of $G$ such that $ ^0R_{\alpha}(G)\ge  ^0R_{\alpha}(H)$ and every quasi vertex $z$ has $d_{H}(z)=2$,  being adjacent in $H$ to two vertices which are not quasi vertices, which implies that $\sum _{v\in V(G)}d_{H}(v)=2(n+k-1).$ By Lemma \ref{aq} $ ^0R_{\alpha}(H)$ is minimum if the degrees of $H$ are almost equal. We deduce that in this case the degrees of $H$ are equal to 2 or to 3. By denoting $n_i$ the number of vertices having degree $i$ we can write $2n_2+3(n-n_2)=2n+2k-2$, which implies $n_2=n-2k+2$ and $n_3=n-n_2=2k-2$ and yields the lower bound.
Consequently,  the minimum of $^0R_{\alpha}(G)$ is reached if and only if there exist $n-2k+2$ vertices (including quasi vertices) of degree $2$ and $2k-2$ vertices of degree $3$ (in this case $H=G$). Such a graph is illustrated in Fig. \ref{fig}. Note that for $k=1$ we have $n_2=n$ and $n_3=0$, hence $G=C_n$, the cycle with $n$ vertices.

\end{proof}

\begin{figure}
	\centering
	\includegraphics[width=0.9\linewidth]{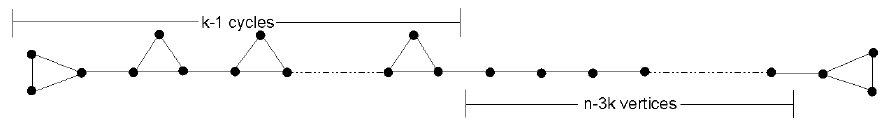}
	\caption{k-generalized quasi tree with almost equal vertices degree.}
	\label{fig}
\end{figure}

\section{Case $0<\alpha<1$}
By similar methods as in preceding sections we can deduce the extremal values of 
$^0R_{\alpha}(G)$ for $0<\alpha <1$ as follows:

\begin{thm}
Let $G\in T_k(n)$, $k\geq 1$ and $n\ge 3$. If $0<\alpha<1$ then
	\[  ^0R_{\alpha}(G)\le  k(n-1)^{\alpha}+2(k+1)^{\alpha}+(n-k-2)(k+2)^{\alpha}.\]
	Equality holds if and only if $G=K_k+P_{n-k}$.
\end{thm}
\begin{thm}
Let $G\in T_k(n)$, where $n\ge 3$ and $k\geq 1$. If $0<\alpha<1$ we have:\\
{\rm a)} If $k=1$ then
	\[^0R_{\alpha}(G)\ge (n-1)^{\alpha}+2^{\alpha+1}+n-3\]
	and equality holds if and only if $G={K_1}\bullet _{u,v}S_{n-1}$, where $u$ is the center of $S_{n-1}$ and $v$ is a pendant vertex of $S_{n-1}$.\\
{\rm b)} If $n\ge 4$ and $k\geq 2$ then \[^0R_{\alpha}(G)\ge (n-2)^{\alpha}+k2^{\alpha}+(k+2)^{\alpha}+n-k-2\]
and equality holds if and only if $G=\overline{K_k}\bullet _{u,v}S_{n-k-2,2}(u,v)$, where $u$ and $v$ are vertices of degree $n-k-2$ and $2$ of $S_{n-k-2,2}(u,v)$, respectively.

\end{thm}

\end{document}